\documentclass[10pt]{article}

\usepackage{latexsym,color,amsmath,amsthm,amssymb,amscd,amsfonts}

\newtheorem{theo}{Theorem}
\newtheorem{cor}{Corollary}
\newtheorem{lem}{Lemma}
\newtheorem{defn}{Definition}
\newtheorem{prop}{Proposition}
\theoremstyle{remark}
\newenvironment{rmk}[1]{{\noindent \bf Remark.} {\it #1}}

\providecommand{\sca}[1]{\langle #1 \rangle}
\providecommand{\e}{\mathrm e}

\newcommand{\ps}{\psi^{\star}_{(s)}}

\def\R{{\mathbb R}}
\def\grad{\nabla}

\def\qed{\hfill $\vcenter{\hrule height .3mm
\hbox {\vrule width .3mm height 2.1mm \kern 2mm \vrule width .3mm
height 2.1mm} \hrule height .3mm}$ \bigskip}

\def\lam{\lambda}

\def\to{\rightarrow}

\def\pmx{\begin{pmatrix}}
\def\emx{\end{pmatrix}}
\def\Hess{\nabla^2}

\def\det{{\rm det}}

\def\R{\mathbb R}

\begin{document}

\title{Functional versions of $L_p$-affine surface area and entropy inequalities.
}
\author{U. Caglar, M. Fradelizi\thanks{partially supported by the project GeMeCoD ANR 2011 BS01 007 01}, O. Gu\'edon\footnotemark[2], J. Lehec, \\
C.  Sch\"utt  and E.  M. Werner 
\thanks{Partially supported by an NSF grant}}

\date{}

\maketitle

\begin{abstract}
In contemporary convex geometry, the rapidly developing $L_p$-Brunn Minkowski theory is a modern analogue of the classical Brunn Minkowski theory.
A cornerstone of this theory is the $L_p$-affine surface area for convex bodies.
Here, we introduce  a functional form of this concept,  for log concave and $s$-concave functions. 
We show that the new  functional form is a generalization of the original $L_p$-affine surface area.
We prove duality relations and affine isoperimetric inequalities for log concave and $s$-concave functions. 
This leads to a new inverse log-Sobolev inequality for $s$-concave densities.

\end{abstract}

\section {Introduction.}

The starting point of this paper is a reverse log-Sobolev inequality for log concave functions due to Artstein, Klartag, Sch\"utt and Werner~\cite{ArtKlarSchuWer}. We first recall the usual log-Sobolev inequality.
Let $\gamma_n$ be the standard Gaussian 
measure on $\R^n$. The log-Sobolev inequality, due to Gross~\cite{Gross} (see also \cite{Federbush, Stam}),
asserts that for every probability measure $\mu$ on $\R^n$
\[ 
H \left( \mu \mid \gamma_n \right) 
  \leq  \frac 1 2  I \left( \mu \mid \gamma_n \right) ,
\]
where $H$ and $I$ denote the relative entropy and Fisher information, respectively, 
\[
\begin{split}
H( \mu \mid \gamma_n ) =  \int_{\R^n} \log \left( \frac{d\mu}{d\gamma_n} \right) \, d \mu,  \hskip 4mm 
I ( \mu \mid \gamma_n )  =
 \int_{\R^n} \left| \nabla \log\left( \frac{d\mu}{d\gamma_n} \right) \right|^2 \, d \mu 
 \end{split}
 \]
and $| \cdot |$ is the Euclidean norm.
It is well known (see for instance~\cite{BakryLedoux}) that this inequality can be slightly improved to
 \begin{equation}
 \label{log-sob-mieux}
 H ( \mu \mid \gamma_n ) 
 \leq \frac {C(\mu)} 2  + \frac n 2  \log \Bigl( 1 + \frac { I ( \mu \mid \gamma_n ) - C(\mu) }{n} \Bigr) ,
 \end{equation}
where
\[
C ( \mu) =   \int_{\R^n} \vert x\vert^2 \, d\mu- n 
\]
is the gap between the second moment of $\mu$ and that of the Gaussian. 
The usual log-Sobolev inequality is recovered using the inequality $\log ( 1+x ) \leq x $. 
Inequality~\eqref{log-sob-mieux} can be written in a more concise way. 
Put $\psi = - \log ( d \mu / dx )$ and let 
\[
 S ( \mu ) =  \int_{\R^n} \psi \, d\mu  = -  H ( \mu \mid dx )=-H ( \mu \mid \gamma_n ) +\frac{C(\mu)}{2}+\frac{n}{2}\log(2\pi \e)
\]
be the Shannon entropy of $\mu$. Then $S(\gamma_n)=\frac{n}{2}\log(2\pi \e)$ so that 
\[
H ( \mu \mid \gamma_n ) -\frac{C(\mu)}{2}=S (  \gamma_n ) - S ( \mu ).
\]
Moreover one has
\[
I ( \mu \mid \gamma_n ) =\int |x-\nabla\psi(x)|^2d\mu=C(\mu)+n+\int\left(|\nabla\psi(x)|^2-2\langle x,\nabla\psi(x)\rangle\right)d\mu.
\]
Hence  inequality~\eqref{log-sob-mieux} is equivalent to
\[
2\  \Big(  S (  \gamma_n ) - S ( \mu ) \Big)
\leq   n  \log \left(  \frac { 2n- 2\int\langle x,\nabla\psi(x)\rangle d\mu+ \int|\nabla\psi(x)|^2d\mu}{n} \right) .
\]
If $\e^{-\psi}$ is $C^2$ on $\R^n$,  then 
$\int\langle x,\nabla\psi(x)\rangle d\mu=n$ and $\int|\nabla\psi(x)|^2d\mu=\int \Delta \psi \, d\mu$ so that  
 inequality~\eqref{log-sob-mieux} is equivalent to
 \[
2\  \Big(  S (  \gamma_n ) - S ( \mu ) \Big)
\leq   n  \log \left(  \frac {  \int_{\R^n} \Delta \psi \, d\mu } n \right) ,
\]
where $\Delta$ is the Laplacian. 

Recall that a measure $\mu$ with density $e^{-\psi}$ with respect to the Lebesgue measure
is called log-concave if 
$\psi \colon \R^n \to \R \cup\{+\infty\}$ is a convex function.
For such log-concave measures
the following reversed form of the 
previous inequality holds.  There, $\nabla^2 \psi$ denotes the Hessian of $\psi$.

\begin{theo}
\label{rev-log-sob}
Let $\mu$ be a log-concave probability measure on $\R^n$, with density $\e^{-\psi}$ with respect to the Lebesgue measure. 
Then 
\[
\int_{\R^n} \log \bigl( \det ( \nabla^2 \psi) \bigr) \, d\mu 
\leq 2\  \Big(  S (  \gamma_n ) - S ( \mu ) \Big) . 
\]
Equality holds if and only if $\mu$ is Gaussian (with any mean and 
any positive definite covariance matrix).
\end{theo}

The  inequality  of Theorem \ref{rev-log-sob} is due to Artstein, Klartag, Sch\"utt and Werner~\cite{ArtKlarSchuWer}, apart from 
the equality case which was left open and smoothness hypotheses which we removed. Their proof is based on affine isoperimetric inequalities
and is pretty technical. 
\par
It is one aim of the present  article  to give a simple and short proof of this theorem including the characterization of equality,  based
on the functional form of the Blaschke-Santal\'o inequality. 
\par
This new approach can be  extended  to a  more general scheme which we  develop in subsequent sections. 
In particular, it  leads to  the definition of {\em functional $L_p$-affine surface area}.  In Theorem \ref{generalduality} and Corollary \ref{funcsurfineq}, we establish, for log concave functions,  their corresponding duality relation and $L_p$-affine isoperimetric inequalities. Those are the counterparts to  the ones that hold for convex bodies. In fact, we show that the $L_p$-affine isoperimetric inequalities for convex bodies can be obtained from the ones for log concave functions. This is explained in section \ref{sec:convexbodies}..
\par
Finally, we generalize the notion of $L_p$-affine isoperimetric surface area to $s$-concave functions for $s>0$. We establish in Theorem \ref{s-generalduality} a duality relation which enables to prove the corresponding $L_p$-affine inequalities and the reverse log-Sobolev inequality for $s$-concave functions.

\subsection{Notations}
For a convex function $\psi: \R^n \to \R \cup \{+\infty\}$, we define $\Omega_\psi$ to be the interior of the convex domain of $\psi$, $\{x \in \R^n, \psi(x) < +\infty\}$. We always consider in this paper convex functions $\psi$ such that $\Omega_\psi \ne \emptyset$. We will use the classical Legendre transform of $\psi$, 
\begin{equation}
\label{Legendre}
\psi^* ( y ) = \sup_x \bigl( \sca{x,y} - \psi (x) \bigr) .
\end{equation} 
In the general case, when $\psi$ is neither  smooth nor strictly convex,
the gradient of $\psi$, denoted  by $\nabla \psi$, exists almost everywhere by Rademacher's theorem (e.g., \cite{Rademacher}),   and a theorem of Alexandrov \cite{Alexandroff} and Busemann and Feller \cite{Buse-Feller} guarantees the existence of its Hessian, denoted  
$\nabla^2 \psi$, almost everywhere in $\Omega_\psi$. We let $X_\psi$  be the set of points of $\Omega_\psi$ at which its Hessian $\nabla^2\psi$ in the sense of Alexandrov exists and is invertible.
Recall also  that 
\[
\psi ( x) + \psi^* (y ) \geq \sca{x,y} 
\] 
for every $x,y \in \R^n$, with equality if and only if
 $x$ is in the domain of $\psi$ and $y\in\partial\psi(x)$, the sub differential 
 of $\psi$ at $x$. In particular
\begin{equation}\label{legendreequality}
 \psi^* ( \nabla \psi (x) ) = \langle x ,\nabla \psi (x) \rangle - \psi (x)  ,\quad \rm{a.e.\ in}\ \Omega_\psi. 
\end{equation}
References about duality of convex functions are \cite{mccann, Rockafellar, SchneiderBook}. We will denote by $|x|$ the Euclidean norm of a vector $x$ in $\R^n.$

\section{A  short proof of the  reverse log-Sobolev inequality}
\label{sec:shortproof}

Let us first recall the form of the functional Blaschke-Santal\'o inequality  \cite{ArtKlarMil, KBallthesis, FradeliziMeyer2007,Lehec2009} that we need.  Let $f,g$ be non-negative integrable 
functions on $\R^n$ satisfying
\[
f ( x ) g ( y ) \leq \e^{- \sca{x,y}} , \quad \forall x,y \in \R^n . 
\]
If $f$ has its barycenter at $0$, which means that $\int xf(x)dx=0$, then
\[
\Bigl( \int_{\R^n} f \, dx \Bigr) \times \Bigl( \int_{\R^n} g \, dx \Bigr) \leq ( 2 \pi )^n . 
\]
There is equality if and only if there exists a positive definite matrix $A$ and $C>0$ such that, a.e. in $\R^n$,
\[
f (x) = C\,  \e^{ - \sca{ A x , x } / 2  } , \quad g(y) = \frac {  \e^{ - \sca{ A^{-1} y , y } / 2 }  } C.
\]
\begin{proof}[Proof of Theorem~\ref{rev-log-sob}]
Without loss of generality, we may assume that the function $\psi$ is lower semi-continuous. Both terms of the inequality are invariant under translations of
the measure $\mu$, so we can assume
that $\mu$ has its barycenter at $0$. 
Then by the functional Santal\'o inequality above
\begin{equation}
\label{step-rev-log-sob}
 \int_{\R^n} \e^{-\psi^* } \, dx \leq (2\pi)^n.
\end{equation}
Let $\Omega_\psi, \Omega_{\psi^*}$ be the interiors of the domains of
$\psi$ and $\psi^*$, respectively. 
If $\psi$ is $\mathcal C^2$-smooth and strictly convex
then the map 
$\nabla \psi \colon \Omega_\psi \to \Omega_{\psi^*}$ is smooth
and bijective. So by the change of variable formula, 
\begin{equation}\label{variablechange1}
\int_{\R^n} e^{-\psi^*(y)} \, dy 
= \int_{\Omega_{\psi^*}} e^{-\psi^*(y)} \, dy  
= \int_{\Omega_\psi} e^{-\psi^*(\nabla \psi (x))} \det ( \nabla^2 \psi (x) ) \, dx . 
\end{equation}
As noted above, in the general case, 
Rademacher's theorem still guarantees the existence of the gradient $\nabla \psi$ of $\psi$ and a theorem of Alexandrov and Busemann and Feller the existence of  its Hessian $\nabla^2 \psi$, almost everywhere 
in $\Omega$, so that both terms of equality~\eqref{variablechange1} make sense.
Although it is clear (take $\psi (x) = \vert x \vert$ in $\R$) that this equality may fail in general, 
a result of McCann~\cite[Corollary~4.3 and Proposition A.1]{mccann} shows that 
\begin{equation}\label{variablechangemccann}
\int_{\Omega_\psi} \e^{-\psi^*(\nabla \psi (x))} \det ( \nabla^2 \psi (x) )\, dx= \int_{X_{\psi^*}}\e^{-\psi^*(y)} \, dy,
\end{equation}
where $X_{\psi^*}$ is the set of vectors of $\Omega_{\psi^*}$ at which $\nabla^2\psi^*$ exists and is invertible.
Together with~\eqref{step-rev-log-sob} we get
\[
 \int_{\Omega_\psi} \e^{- \psi^* ( \nabla \psi (x) ) } \det ( \nabla^2 \psi (x) ) \, dx  \leq (2\pi)^n .
\]
With \eqref{legendreequality}, the previous inequality thus becomes
\[
 \int_{\Omega_\psi} \e^{- \sca{x,\nabla \psi (x) } +  \psi (x) } 
\det ( \nabla^2 \psi (x) ) \, dx  \leq (2\pi)^n  , 
\]
which can be rewritten as
\begin{equation}
\label{variablechange}
 \int_{\R^n} \e^{- \sca{x,\nabla \psi (x) } + 2  \psi (x) } 
\det ( \nabla^2 \psi (x) ) \, d\mu  \leq (2\pi)^n  . 
\end{equation}
Taking the logarithm and using Jensen's inequality (recall that $\mu$
is assumed to be a probability measure) we obtain
\[
-  \int_{\R^n} \sca{x,\nabla \psi (x) } \, d\mu   + 2 S ( \mu ) +
\int_{\R^n} \log\bigl( \det ( \nabla^2 \psi ) \bigr)  \,  d\mu   \leq n \log (2\pi)  .
\]
We will need some version of the Gauss-Green (or Stokes) formula and refer to \cite{ChenTorresZiemer} for general references and recent results on this subject. 
Let $v$ be the  vector flow $v(x) = e^{-\psi(x)} x$. By convexity and lower semi-continuity of $\psi$, it is continuous and locally Lispchitz on 
$\overline{\Omega_\psi}$. Assume first that $\Omega_\psi$ is bounded.  Then by the Gauss-Green formula  \cite{DeGiorgi, Federer}, we have 
\[
\int_{\Omega_\psi}  \mathrm{div}( v(x)) dx = \int_{\partial \Omega_\psi}\langle v(x), N_{\Omega_\psi}(x)\rangle d\sigma_{\Omega_\psi},
\]
where $N_{\Omega_\psi}(x)$ is an exterior normal to the convex set $\Omega_\psi$ at the point $x$ and $\sigma_{\Omega_\psi}$ is the surface area measure on $\partial \Omega_\psi$. Hence 
\begin{align*}
\int_{\R^n} \sca{x,\nabla \psi (x) } \, d\mu  & = \int_{\Omega_\psi} \sca{x,\nabla \psi (x) } e^{-\psi(x)}  dx 
\\
& = \int_{\Omega_\psi}  \mathrm{div} ( x ) e^{-\psi(x)} dx  - \int_{\partial \Omega_\psi}\langle x, N_{\Omega_\psi}(x)\rangle \e^{-\psi(x)}d\sigma_{\Omega_\psi}.
\end{align*}
This formula holds true for unbounded domain $\Omega_\psi$  by a simple truncation argument and by the fast decay of $\log$-concave integrable functions.  Since $\Omega_\psi$ is convex, the barycenter $0$ of $\mu$ is in $\Omega_\psi$. Thus $\langle x, N_{\Omega_\psi}(x)\rangle\ge 0$ for every $x\in\partial\Omega_\psi$ and $ \mathrm{div}(x)=n$ hence 
\[
\int_{\R^n} \sca{x,\nabla \psi (x) } \, d\mu \le n.
\]
This finishes the proof of the inequality.
Let us move on to the equality case. 
It is easily checked that there is equality 
 in Theorem~\ref{rev-log-sob} for Gaussian measures. 
On the other hand, the above proof shows that if $\mu$ 
satisfies the equality case, 
then there must be equality in~\eqref{step-rev-log-sob}. Then,
by the equality case of the functional Santal\'o inequality, $\mu$ is Gaussian.
\end{proof}

\section{A functional $L_p$-affine surface area.}

\subsection{General theorems.}
We first present a definition  that generalizes the notion of $L_p$-affine surface area of  convex bodies  to a  functional setting.
Generalizations of a different nature were given in \cite{CaglarWerner} and \cite{ CaglarWerner2}.
\begin{defn}
\label{def:log}
For $F_1, F_2 \colon \R\to (0,+\infty)$ and  $\lambda\in\R$, we define
\begin{equation}\label{general}
as_{\lambda}(F_1, F_2, \psi) = \int_{X_\psi} \Big(F_1 (\psi(x))\Big)^{1-\lambda} \Big(F_2(\langle x, \nabla \psi(x)\rangle - \psi(x))\Big)^{\lambda} 
\Big(\det \,  \nabla^2  \psi(x)\Big)^{\lambda} dx.
\end{equation}
\end{defn}
Since $\det (\nabla^2 \psi(x))=0$ outside $X_\psi$, the integral may be taken on $\Omega_\psi$ for $\lambda>0$.
Definition \ref{def:log} is motivated by two important facts. Firstly, we can prove that  for a particular choice of $F_1$, $F_2$ and $\psi$  it fits with the usual $L_p$-affine surface area of a convex body. This is the content of Theorem \ref{norm}. Secondly, in the  case of log-concave functions,  for $F_1(t)=F_2(t)=e^{-t}$ the functional affine surface area $as_1(F_1,F_2, \psi)$  becomes 
\[
as_1(F_1, F_2, \psi) = \int_{X_\psi} e^{-\psi^*(\nabla \psi(x)} \det \,  \nabla^2  \psi(x) dx = \int_{\Omega_\psi} e^{-\psi^*(\nabla \psi(x)} \det \,  \nabla^2  \psi(x) dx
\]
and is of particular interest. This is illustrated in subsection \ref{application}.
\par
Our main result is the duality formula of Theorem \ref{generalduality}. A special case is the identity \eqref{variablechangemccann} which was the starting point of the  short proof of the reverse log-Sobolev inequality presented in the Section \ref{sec:shortproof}. 
\par
Notice  also that for any linear invertible map $A$ on $\R^n$, one has 
\begin{equation}\label{linearmap}
as_\lambda(F_1,F_2,\psi \circ A) = |\det A|^{2\lambda-1}as_\lambda (F_1,F_2,\psi),
\end{equation}
which corresponds to an $SL(n)$ invariance with a homogeneity of degree $(2 \lambda -1)$.
This  is easily checked using that $\nabla_x(\psi\circ A)=A^t\nabla_{Ax}\psi$ and $\Hess_x (\psi\circ A)=A^t\Hess_{Ax}\psi A$.

We shall use  Corollary 4.3 and Proposition A.1 of \cite{mccann}, where McCann showed a  general  change of variable formula,  namely  for every Borel function $f : \R^n \to \R_+$, 
\begin{equation}
\label{mccannbis}
\int_{X_\psi} f(\nabla \psi (x)) \det \Hess \psi(x) dx = \int_{X_{\psi^*}} f(y) dy.
\end{equation}
The same holds true for every integrable function $f: \R^n \to \R$. Identity (\ref{mccannbis})  is obvious when $\psi$ satisfies some regularity assumptions, like $C^2$. It suffices to make the change of variable $y = \nabla \psi(x)$. The proofs are however more delicate in a general setting. 
\par
We establish the following duality relation.

\begin{theo}\label{generalduality}
Let $\lambda\in\R$, let $F_1, F_2\colon \R\to \R_+$ and let $\psi : \R^n \to \R \cup \{+\infty\}$ be convex. If $\lambda < 0$ or $\lambda > 1$, assume moreover that $F_1 \circ \psi > 0$ on $X_\psi$ and $F_2 \circ \psi^* > 0$ on $X_{\psi^*}$.
Then
\[
as_{\lambda}(F_1, F_2, \psi)  = as_{1-\lambda}(F_2, F_1, \psi^*).
\]
\end{theo}
\begin{proof}
Without loss of generality, we can assume that $\psi$ is lower semi-continuous so that $\psi = (\psi^*)^*$.
By (\ref{legendreequality}), 
\[
as_{\lambda}(F_1, F_2, \psi) = \int_{X_\psi} (F_1\circ\psi(x))^{1-\lambda} (F_2\circ\psi^*(\nabla \psi(x)))^{\lambda} 
(\det \,\nabla^2 \psi(x))^{\lambda} dx.
\]
By Proposition A.1 in \cite{mccann}, 
\[
x = \nabla \psi^* \circ \nabla \psi(x) 
\ \hbox{ and } \ 
\Hess \psi^*(\nabla \psi(x)) = (\Hess \psi (x))^{-1},   \quad \forall x \in X_\psi, 
\]
so that $as_{\lambda}(F_1, F_2, \psi)$ is equal to 
\[
 \int_{X_\psi} (F_1\circ\psi\circ\nabla \psi^* ( \nabla \psi (x))^{1-\lambda} (F_2\circ\psi^*(\nabla \psi(x)))^{\lambda} 
(\det \,\nabla^2 \psi^*(\nabla \psi(x)))^{1-\lambda} \det \,\nabla^2 \psi(x)dx.
\]
With  (\ref{mccannbis}), we get that 
\[
as_{\lambda}(F_1, F_2, \psi) = \int_{X_{\psi^*}} (F_1\circ\psi\circ\nabla \psi^*(y)^{1-\lambda} (F_2\circ\psi^*(y))^{\lambda} 
(\det \,\nabla^2 \psi^*(y))^{1-\lambda}  dy.
\]
We conclude the proof using (\ref{legendreequality}) with $\psi^*$ and $(\psi^*)^* = \psi$. 
\end{proof}
\begin{cor}\label{generalholder}
The function  $\lambda\mapsto \log(as_{\lambda}(F_1, F_2, \psi))$ is convex on $\R$. Moreover,  
\begin{equation*}
\forall \lambda \in [0,1],  \hskip 2mm 
 as_{\lambda}(F_1, F_2, \psi) \le \left(\int_{X_\psi} F_1 \circ \psi\right)^{1-\lambda}\left(\int_{X_{\psi^*}} F_2 \circ \psi^*\right)^{\lambda}.
\end{equation*}
Equality holds trivially if $\lam=0 $ and $ \lam=1$.
\begin{equation*}
\forall \lambda \notin [0,1],  \hskip 2mm
 as_{\lambda}(F_1, F_2, \psi)\ge \left(\int_{X_\psi} F_1\circ \psi \right)^{1-\lambda}\left(\int_{X_{\psi^*}} F_2 \circ \psi^* \right)^{\lambda}.
\end{equation*}

\end{cor}
\par
\begin{proof}
The convexity of $\lambda\mapsto \log(as_{\lambda}(F_1, F_2, \psi))$ is a consequence of H\"older inequality. 
For the inequalities we use H\"older inequality  and  also  the duality relation of  Theorem \ref{generalduality} with  $\lambda =1$,
$as_1(F_1, F_2, \psi)=as_0(F_2,F_1,\psi^*)=\int_{X_{\psi^*}} F_2 \circ \psi^*.$
\end{proof}
We define the non-increasing function $F:\R \to\R_+$ by 
\begin{equation}
\label{eq:defofF}
F(t)= \sup_{\frac{t_1+ t_2}{2} \ge t} \sqrt{F_1(t_1) F_2 ( t_2 )} . 
\end{equation}
Notice that if $F_1=F_2$ is a log-concave,  non-increasing function then $F=F_1=F_2$. 
\begin{cor}\label{generalaffineineq}
Let $F_1, F_2 \colon \R \to \R_+$, let $\psi : \R^n \to \R \cup \{+\infty\}$ be a convex function. Then there exists $z\in\R^n$ such that
\[
\forall \lambda \in [0,1/2], 
 \quad
as_{\lambda}(F_1, F_2, \psi_z) 
\le  \left( \int_{\R^n} F \left(\frac{|x |^2}{2}\right) \, d x \right)^{2\lambda} 
\left(\int_{X_\psi} F_1 \circ \psi \right)^{1-2\lambda} 
\]
Equality holds trivially if $\lam=0 $.
If $F_1 \circ \psi > 0$ on $X_{\psi}$ and $F_2 \circ \psi^* > 0$ on $X_{\psi^*}$ then 
\[
\forall \lam <0,   \quad
as_{\lambda}(F_1, F_2, \psi_z) 
\ge  \left( \int_{\R^n} F \left(\frac{|x |^2}{2}\right) \, d x \right)^{2\lambda} 
\left(\int_{X_\psi} F_1 \circ \psi \right)^{1-2\lambda}, 
\]
where $\psi_z(x) = \psi(z+x)$.
\\
If $F$ is decreasing,  $\lambda \ne 0 $ and $\int_{X_\psi} F_1 \circ \psi \ne 0$,  then there is equality in each of these inequalities if and only if  there exists $ c \in \R_+$, $a \in \R$ and a positive definite matrix $A$ such that, for every $x\in\R^n$ and $t\ge 0$, 
\[
\psi_z(x)=\langle Ax,x\rangle+a,\quad F_1(t+a)=c\  F(t)\quad{\rm and}\quad F_2(t-a)=\frac{F(t)}{c}.
\]
\end{cor}
\begin{rmk}{
(i) Notice that if $\psi$ is even  then one may choose $z=0$.
\\
(ii)
Moreover, for $\lam > 1/2$, we deduce from the duality relation proved in Theorem \ref{generalduality} that the same inequalities hold true exchanging $F_1$ and $F_2$, $\psi$ and $\psi^*$ and that the equality case is characterized for $\lam \ne 1$. 

}
\end{rmk}
\begin{proof}
We recall a general form of the functional Blaschke-Santal\'o inequality  \cite{FradeliziMeyer2007,Lehec2009b}.  Let $f$ be a non-negative integrable 
function on $\R^n$. There exists $z_0\in\R^n$ such that for every $\rho:\R_+\to\R_+$ and every $g : \R^n\to\R_+$ satisfying
\begin{equation}\label{hypsantalo}
f (z_0+ x ) g ( y ) \le \left(\rho(\sca{x,y})\right)^2 ,  
\end{equation}
for every $x,y \in \R^n$ with $\sca{x,y}>0$, we have 
\begin{equation}\label{concsantalo}
\int_{\R^n} f \, dx  \int_{\R^n} g \, dx  \leq \left(\int_{\R^n}\rho(|x|^2)dx\right)^2 . 
\end{equation}
If $f$ is even, a result of Ball \cite{KBallthesis} asserts that  one may choose $z_0 = 0$. 
Moreover, if there exists $g$ satisfying \eqref{hypsantalo} and equality holds in \eqref{concsantalo},  then there exists $c>0$ and an invertible $T$,  such that for every $x\in\R^n$, 
\begin{equation}
\label{Santalo-equality}
f(z_0+x)=c\rho\left(|Tx|^2\right) \quad{\rm and}\quad g(y)=\frac{1}{c}\rho\left(|T^{-1}x|^2\right).
\end{equation}
For $z\in\R^n$, let us denote $\psi_z^*=(\psi_z)^*$. Since $F$ is non-increasing, we have by (\ref{Legendre}), for every $x,y,z\in \R^n$ such that $\langle x,y \rangle>0$,
\[
F_1 ( \psi_z (x ) ) F_2 ( \psi_z^* (y ) ) 
\leq 
F^2 \left(\frac{ \psi_z ( x) +  \psi_z^* (y) }{2}\right)
\leq F^2 \left( \frac{\langle x,y \rangle}{2} \right) .
\]
By  the functional Blaschke-Santal\'o inequality  there exists $z_0\in\R^n$ such that 
\begin{equation}
\label{eq:BSfunctional}
\left( \int F_1 \circ \psi \right) \left( \int F_2 \circ \psi_{z_0}^* \right)
\leq 
\left( \int_{\R^n} F\left(\frac{|x |^2}{2}\right)  \, d x \right)^2 .
\end{equation}
Applying Corollary \ref{generalholder} to $\psi_{z_0}$,  we deduce that for $\lambda \in [0,1]$, 
 \begin{align*}
as_{\lambda}(F_1, F_2, \psi_{z_0}) 
& \le 
\left(\int_{X_\psi} F_1 \circ \psi \right)^{1-\lambda}\left(\int_{X_{\psi_{z_0}^*}} F_2 \circ \psi_{z_0}^*\right)^{\lambda} 
\\
& \le
\ \left( \int_{\R^n} F \left(\frac{|x |^2}{2}\right) \, d x \right)^{2\lambda} 
\left(\int_{X_{\psi} }F_1 \circ \psi \right)^{1-2\lambda}. 
\end{align*}
For $\lambda < 0$ we deduce from (\ref{eq:BSfunctional}) that 
\[
\left( \int_{X_{\psi_{}} }F_1 \circ \psi \right)^\lambda \left( \int_{X_{\psi_{z_0}^*}} F_2 \circ \psi_{z_0}^* \right)^\lambda
\geq 
\left( \int_{\R^n} F\left(\frac{|x |^2}{2}\right)  \, d x \right)^{2\lambda} 
\]
and we conclude by using the second part of Corollary  \ref{generalholder}.

To characterize the equality case, we suppose that $\int_{X_\psi} F_1 \circ \psi \ne 0$ which means that the expressions are not identically zero in the inequality. For $\lambda \neq 0$, if there is equality in one of the inequalities of Corollary \ref{generalaffineineq}, it follows from the proof that we have equality in the functional  Blaschke-Santal\'o inequality. Thus by (\ref{Santalo-equality}), 
there exists $c>0$ and an invertible matrix $T$,  such that for every $x\in\R^n$, 
\[
F_1\circ\psi_{z_0}(x)=c\  F\left(\frac{|Tx|^2}{2}\right) \quad{\rm and}\quad F_2\circ\psi_{z_0}^*(x)=\frac{1}{c}F\left(\frac{|T^{-1}x|^2}{2}\right).
\]
Let us define $\varphi(x)=\psi(T^{-1}x+z_0)$. Then we have 
\begin{eqnarray}\label{equalitycase}
F_1\circ\varphi(x)=c \ F\left(\frac{|x|^2}{2}\right) \quad{\rm and}\quad F_2\circ\varphi^*(x)=\frac{1}{c}F\left(\frac{|x|^2}{2}\right).
\end{eqnarray}
Hence 
\[
F\left(\frac{|x|^2}{2}\right)=\sqrt{F_1\circ\varphi(x)F_2\circ\varphi^*(x)}\le F\left(\frac{\varphi(x)+\varphi^*(x)}{2}\right)\le F\left(\frac{|x|^2}{2}\right).
\]
Since $F$ is decreasing,  we deduce that $\varphi(x)+\varphi^*(x)=|x|^2$. It is classical that this  implies that $\varphi(x)=|x|^2/2+a$.  See for example the argument given in the proof of Theorem~8 in \cite{FradeliziMeyer2007}. Defining $A=T^*T/2$,  we get  that $\psi_{z_0}(x)=\langle Ax,x\rangle+a$, for every $x\in\R^n$. From (\ref{equalitycase}) we deduce that for every $t\ge0$ 
\[
F_1(t+a)=c\  F(t)\quad{\rm and}\quad F_2(t-a)=\frac{1}{c}F(t).
\] 
Therefore all the conditions of the theorem are proved. Reciprocally, if these conditions are fulfilled,  a simple computation shows that there is  equality.
\end{proof}

\subsection{Application to particular functions: the log-concave case.} \label{application}

We define $F_1$ and $F_2$ on $\R$ by
  $F_1(t)=F_2(t) = e^{-t}$ ; then $F ( t ) = e^{-t}$ as well and we use the simplified notation 
\begin{equation}\label{asa}
as_\lam(\psi) = as_\lambda(e^{-t}, e^{-t},\psi) = \int_{X_\psi}e^{(2\lam-1)\psi(x)-\lam\langle x, \nabla\psi(x)\rangle}\left(\det \, \Hess \psi (x)\right)^\lam dx.
\end{equation}
Again, as before,  we can replace $X_\psi$ by $\Omega_\psi$ for $\lam >0$.
Observe that for the Euclidean norm $|\cdot|$,
\begin{equation}\label{Euklid}
as_\lam\left(\frac{|\cdot|^2}{2}\right) = \left(2 \pi\right)^\frac{n}{2}.
\end{equation}
Moreover, it is not difficult to see (see e.g., \cite{CaglarWerner}) that for any $\lam \in \R$, $\psi \mapsto as_\lambda(\psi)$  is a valuation on the set of convex functions $\psi$, i.e.,
if $\min
(\psi_1,\psi_2)$ is convex,  then
\[ as_\lam(\psi_1)+ as_\lam(\psi_2) =  as_\lam(\max (\psi_1, \psi_2)) + as_\lam(\min (\psi_1, \psi_2)), 
\]
and it is homogeneous of degree $(2 \lam -1)n$,  since  we have by (\ref{linearmap}) 
 for any linear invertible map $A$ on $\R^n$, for all convex 
$\psi$
$$
as_\lambda(\psi \circ A) =  |\det A|^{2\lambda-1} as_\lambda (\psi).
$$
For convex bodies with the origin in their interiors, such upper semi-continuous valuations were characterized as  $L_p$-affine surface areas  in \cite{Ludwig-Reitzner1999} and  \cite{Ludwig-Reitzner} which motivated us to call $as_\lam (\psi)$ the $L_\lam$-affine surface area of $\psi$.
This is further justified by Theorem \ref{norm}  of the next section (where we also give the definition of $L_p$-affine surface area for convex bodies), and by  the identity (\ref{motiv1}) of Section \ref{asa-sconcave}.

From Theorem~\ref{generalduality} and Corollary \ref{generalholder} we get that $\lam\mapsto \log\left(as_\lam(\psi)\right)$ is convex and that
\begin{equation}
\label{asa-dual}
\forall \lam \in \R, \ as_\lam(\psi)=as_{1-\lam}(\psi^*).
\end{equation}
The following isoperimetric inequalities are a direct consequence of Corollary~\ref{generalaffineineq} and a result of  \cite{Lehec2009b} which says  that the Santal\' o point $z_0$ in the functional Blaschke-Santal\'o inequality \eqref{hypsantalo} can be taken equal to $0$ when $\int xe^{-\psi(x)}dx=0$ or $\int x e^{-\psi^*(x)} dx = 0$.
\begin{cor}\label{funcsurfineq}
Let $\psi :\R^n\to \R\cup\{+\infty\}$ be a  convex function such that  $\int xe^{-\psi(x)}dx=0$ or $\int x e^{-\psi^*(x)} dx = 0$. 
Then
\[
\begin{split}
\forall \lambda \in [0,1/2], 
& \quad as_\lam(\psi)\le (2\pi)^{n\lam}\left(\int_{X_\psi} e^{-\psi}\right)^{1-2\lam},  
\\
\forall \lam\in (- \infty, 0],  &\quad as_\lam(\psi)\ge (2\pi)^{n\lam}\left(\int_{X_{\psi}}e^{-\psi}\right)^{1-2\lam}. 
\end{split}
\]
Equality holds in both inequalities for $\lambda\neq 0$, if and only if there exists $a\in\R$ and a positive definite matrix $A$ such that $\psi(x)=\langle Ax,x\rangle+a$, for every $x\in\R^n$.
\end{cor}
\noindent
\begin{rmk} {(i) To emphasize the isoperimetric character of these inequalities, note that with (\ref{Euklid}), the inequalities are
equivalent to 
\[
\forall \lam \in [0, 1/2], \quad
\frac{as_\lam(\psi)}{as_\lam\left(\frac{|\cdot|^2}{2}\right)} \leq \left( \frac{\int_{X_\psi} e^{-\psi}}{\int e^{-\frac{|\cdot|^2}{2}}}\right)^{1-2\lam}
\]
and
\[
\forall \lam < 0, \quad
\frac{as_\lam(\psi)}{as_\lam\left(\frac{|\cdot|^2}{2}\right)} \geq \left( \frac{\int_{X_\psi} e^{-\psi}}{\int e^{-\frac{|\cdot|^2}{2}}}\right)^{1-2\lam}.
\]
\par
\noindent
(ii) It follows from Corollary \ref{funcsurfineq} and the functional Blaschke Santal\'o inequality that 
\begin{equation*}
\forall \lam \in [0, 1/2], \quad
as_\lam(\psi) as_\lam(\psi^*)  \leq \left(2 \pi\right)^n.
\end{equation*}
}
\end{rmk}
There are several  other direct consequences of Corollary \ref{funcsurfineq} that should be noticed. 
As observed already,  we have 
for every $\lam \in (0, 1/2]$, 
\[
as_\lam(\psi) = \int _{\Omega_\psi} e^{(2\lam-1)\psi(x)-\lam\langle x, \nabla\psi(x)\rangle}\left(\det \, \Hess \psi (x)\right)^\lam dx.
\]
Since $\int_{X_\psi} e^{-\psi} \le \int e^{-\psi}$ we deduce from  Corollary \ref{funcsurfineq} that for any $\lam \in (0, 1/2]$,
\begin{equation}
\label{eq:classic-isoperimetry}
\int _{\Omega_\psi} e^{(2\lam-1)\psi(x)-\lam\langle x, \nabla\psi(x)\rangle}\left(\det \, \Hess \psi (x)\right)^\lam dx
\le 
(2\pi)^{n\lam}\left(\int e^{-\psi }\right)^{1-2\lam}.
\end{equation}
This inequality holds trivially true also for $\lam = 0$. 
Moreover, by Theorem \ref{generalduality}, we know that $as_\lam(\psi) = as_{1 - \lam}(\psi^*)$. Since the inequalities of Corollary \ref{funcsurfineq} are also valid when $\int x e^{-\psi^*(x)} dx = 0$, we deduce from (\ref{eq:classic-isoperimetry}) that if $\lam \in [1/2, 1]$, 
\begin{align*}
\int _{\Omega_\psi} e^{(2\lam-1)\psi(x)-\lam\langle x, \nabla\psi(x)\rangle}\left(\det \, \Hess \psi (x)\right)^\lam  dx & = as_\lam(\psi) 
\\
&= as_{1 - \lam}(\psi^*) 
\le 
(2\pi)^{n(1-\lam)}\left(\int e^{-\psi^*}\right)^{2\lam - 1}.
\end{align*}
By the Blaschke-Santal\'o functional inequality (see (\ref{eq:BSfunctional})), we know that $\int e^{-\psi} \int e^{-\psi^*} \le (2\pi)^n$ and we conclude that
for all $\lam \in [1/2, 1]$,
\[
\int _{\Omega_\psi}e^{(2\lam-1)\psi(x)-\lam\langle x, \nabla\psi(x)\rangle}\left(\det \, \Hess \psi (x)\right)^\lam  dx
\le
(2\pi)^{n\lam}\left(\int e^{-\psi}\right)^{1-2\lam}.
\]
For $\lam< 0$ or $\lam> 1$, an important case concerns  $C^2$ convex functions $\psi$. In such a situation $X_\psi =  \Omega_\psi$ and 
$X_{\psi^*} = \Omega_{\psi^*}$ and we deduce from Corollary \ref{generalaffineineq} that for all $\lam< 0$,
\[
\int_{\Omega_\psi} e^{(2\lam-1)\psi(x)-\lam\langle x, \nabla\psi(x)\rangle}\left(\det \, \Hess \psi (x)\right)^\lam dx
\ge 
(2\pi)^{n\lam}\left(\int e^{-\psi}\right)^{1-2\lam}.
\]
For all $\lam> 1$, we go back to Corollary \ref{generalholder} and deduce that 
\[
\int_{\Omega_\psi} e^{(2\lam-1)\psi(x)-\lam\langle x, \nabla\psi(x)\rangle}\left(\det \, \Hess \psi (x)\right)^\lam dx
= as_{\lam}(\psi) 
\ge 
\left(\int e^{-\psi}\right)^{1 -\lam}
\left(\int e^{-\psi^*}\right)^{\lam}.
\]
By the asymptotic functional reverse Santal\'o inequality \cite{FradeliziMeyer2008} (see also \cite{KlartagMilman} in the even case),  there exists a constant $c >0$ such that $\int e^{-\psi} \int e^{-\psi^*} \ge c^n$.  Therefore,  for all $\lam> 1$,
\[
\int _{\Omega_\psi} e^{(2\lam-1)\psi(x)-\lam\langle x, \nabla\psi(x)\rangle}\left(\det \, \Hess \psi (x)\right)^\lam dx
\ge 
c^{n\lam}\left(\int e^{-\psi}\right)^{1-2\lam}.
\]
We have proved 
\begin{cor}
\label{cor:classic}
Let $\psi :\R^n\to \R\cup\{+\infty\}$ be a proper convex function such that  $\int xe^{-\psi(x)}dx=0$ or $\int x e^{-\psi^*(x)} dx = 0$. 
Then 
\[
\forall \lam \in [0,1], \int _{\Omega_\psi}e^{(2\lam-1)\psi(x)-\lam\langle x, \nabla\psi(x)\rangle}\left(\det \, \Hess \psi (x)\right)^\lam dx
\le 
(2\pi)^{n\lam}\left(\int e^{-\psi}\right)^{1-2\lam}, 
\]
Moreover,  if $\psi \in C^2(\Omega_\psi)$,
\[
\forall \lam < 0, 
\int_{\Omega_\psi} e^{(2\lam-1)\psi(x)-\lam\langle x, \nabla\psi(x)\rangle}\left(\det \, \Hess \psi (x)\right)^\lam dx
\ge 
(2\pi)^{n\lam}\left(\int e^{-\psi}\right)^{1-2\lam}
\]
and there exists an absolute constant $c > 0$ such that 
\[
\forall \lam > 1,
\int _{\Omega_\psi} e^{(2\lam-1)\psi(x)-\lam\langle x, \nabla\psi(x)\rangle}\left(\det \, \Hess \psi (x)\right)^\lam dx
\ge 
c^{n\lam}\left(\int e^{-\psi}\right)^{1-2\lam}.
\]
\end{cor}
These are the complete analogues of the $L_p$-affine surface area inequalities due to \cite{Lutwak1996, Hug1996, SW2004} and this will be discussed in more details in the next subsection.

\subsection{The case of convex bodies.}
\label{sec:convexbodies}
We continue to study the case $F_1(t) = F_2(t) = e^{-t}$. 
Additionally, we consider the case of 2-homogeneous proper convex functions $\psi$, that is $\psi(\lambda x) = \lambda^2 \psi(x)$ for any $\lambda \in \R_+$ and $x \in \R^n$. Such functions $\psi$ are necessarily (and this is obviously sufficient) of the form $\psi(x) = \|x\|_K^2/2$ for a certain convex body $K$ with $0$ in its interior. 
Here, $\| \cdot\|_K$ is the gauge function the  convex body $K$, 
$$
\|x\|_K = \min\{ \alpha\geq 0: \ x \in  \alpha K\} = \max_{y \in K^\circ} \langle x, y \rangle = h_{K^\circ} (x).
$$ 
Differentiating with respect to $\lam$ at $\lam=1$, we get  
$$\langle x, \nabla \psi(x) \rangle = 2 \psi(x). $$ 
Thus  for  $2$-homogeneous functions $\psi$,   formula (\ref{asa}) further simplifies to 
\begin{equation}\label{2homoasa}
as_\lam(\psi)=\int_{X_\psi}\left(\det \, \Hess  \psi (x)\right)^\lam e^{-\psi(x)}dx, 
\end{equation}
where $X_\psi$ is the positive cone generated by the points of $\partial K$ where the Gauss curvature is strictly positive.
The following theorem indicates why we call $as_\lambda(\psi)$ the $L_\lambda$-affine surface area of $\psi$.
First we recall that  for  $p \in \mathbb{R}$, $p \neq -n$, the $L_p$-affine surface area for  a convex body
$K$   in $\mathbb{R}^n$ with the origin in its interior is defined  \cite{Hug1996, Lutwak1996, SW2004} 
as 
\begin{equation}\label{asp-K}
as_p(K) = \int_{\partial K} \frac{ \kappa_K(x)^\frac{p}{n+p}}{\langle x, N_K(x) \rangle^\frac{n(p-1)}{n+p}} d\mu_K(x).
\end{equation}
Here, $N_K(x)$ is the outer unit normal to the boundary $\partial K$ in the boundary point $x$, $\mu_K$ is the usual surface area measure on $\partial K$  and $\kappa_K(x)$ is the Gauss curvature in $x$. We denote by $(\partial K)_+$ the points of $\partial K$ where the Gauss curvature is strictly positive.
\begin{theo} \label{norm}
Let $K$ be  a convex body in $\mathbb{R}^n$ containing the origin in its interior. For any $p \geq 0$, let $\lam=\frac{p}{n+p}$. Then 
\[
as_\lam\left(\frac{\|\cdot\|_K^2}{2}\right) = \frac{(2\pi)^\frac{n}{2}}{n|B_2^n|} \  as_p(K).
\]
Moreover, if $(\partial K)_+$ has full Lebesgue measure in $\partial K$,  then the same relation holds true for every $p \ne -n$.
\end{theo}
\begin{rmk}{  For all $p$, $as_p(B^n_2)= n |B^n_2|$. Therefore, together with (\ref{Euklid}), the identity given in the theorem  can be written as 
\[
\frac{as_\lam\left(\frac{\|\cdot\|_K^2}{2}\right)}{as_\lam\left(\frac{|\cdot|^2}{2}\right)} = \frac{as_p(K)}{as_p(B^n_2)}.
\]
}
\end{rmk}
We will need the following technical lemma.
 \begin{lem} \label{Hesse}
Let $K$ be a convex body in $\mathbb{R}^n$ with the origin in its interior and let 
$\psi(x) = \frac{1}{2} \Vert x \Vert_{K}^2$.
Then for all $x\in (\partial K)_+$, 
\[
\det  \,( \Hess \psi(x) ) = \frac{ \kappa_K (x) }{ \Vert G_K(x) \Vert_{K^\circ}^{n+1} } ,
\]
where  
$G_K \colon (\partial K)_+ \to \mathbb S^{n-1}$ is the Gauss map. 
\end{lem}
\begin{proof} 
Let us fix $x \in (\partial K)_+$. 
The differential $d_x G_K$of $G_K$ at $x\in \partial K$ 
is a linear map from the tangent space $T_x ( \partial K)$ to $T_{G_K(x)} ( \mathbb S^{n-1} )$. 
We can identify both spaces with $G_K(x)^\perp$ and view $d_x G_K$ as a linear operator on 
$G_K(x)^\perp$. Then by definition (see e.g., \cite{SchneiderBook})
\[
\kappa_K ( x ) = \det \, ( d_x G_K(x) ) . 
\] 
Let $f \colon x\in \R^n \mapsto   \Vert x \Vert_K$. For all $x\neq 0$, consider $N_K$, the $0$-homogeneous extension of $G_K$,  defined by $N_K ( x) = G_K \left( \frac{x}{ \Vert x \Vert_K } \right)$.
Then
\[
\nabla f(x) = \frac{ N_K (x) }{ \Vert N_K (x) \Vert _{K^\circ}}.
\] 
Using the identity 
\[
N_{K^\circ}  ( N_K ( x) ) = \frac{x}{\Vert x \Vert_2},
\]
we get 
\[
\Hess f(x) = \frac{ d_x  N_K }{ \Vert N_K (x) \Vert_{K^\circ} } - 
\frac{ ( (d_x N_K)^T x ) \otimes N_K (x) }{  \Vert N_K (x) \Vert^2_{ K^\circ } } 
\]
Therefore, if we put $A = d_x N_K$, $ u= N_K (x)$ and $a = \Vert N_K (x)\Vert_{K^\circ}$, 
\[
\Hess \psi(x) = \Hess ( f^2(x) /2 ) =
\frac A a -
\frac{  (A^T x) \otimes u }{  a^2 } 
+ \frac{ u \otimes u }{ a^2 } . 
\]
Let $B = d_x G_K$. Then $A y = B y$ for every $y \in u^\perp$. Since $N_K$ is $0$-homogeneous,   $A x = 0$.  Thus 
\[
A u = - \frac{ B x_\perp } {\langle x , u \rangle } = - \frac{ B x_\perp } a 
\]
where $x_\perp = x - \langle x, u \rangle u  \in u^\perp$. 
Also, as  $B$ is self-adjoint (see e.g. \cite{SchneiderBook}), 
\[
\begin{split}
\langle  A^T x , y \rangle 
&  = \langle x , B  y \rangle = \langle B x_\perp , y \rangle , \quad \forall y \in u^\perp  \\
\langle A^T x , u \rangle 
&  = - \frac 1 a \langle x , B x_\perp \rangle  = - \frac 1 a  \langle x_\perp , B x_\perp \rangle .
\end{split}
\]
The previous computations show that in a basis adapted to the decomposition 
$\R^n = \mathrm{span} (u) + u^\perp$, 
we have
\[
\Hess\psi (x) =
\frac 1 {a^3}
\left[
\begin{array}{cc}
a + \langle x_\perp , B x_\perp \rangle & - a ( B x_\perp )^T \\
- a B x_\perp & a^2 B
\end{array}
\right] .
\] 
Observe that
\[
\Hess \psi (x)
=
\frac 1 {a^3}
\left[
\begin{array}{cc}
a  & - ( B x_\perp )^T \\
0 & a B
\end{array}
\right] 
\times
\left[
\begin{array}{cc}
1  & 0  \\
- x_\perp & a\, \mathrm{id}_{n-1}
\end{array}
\right] .
\] 
Therefore, 
\[
\det \, ( \Hess \psi(x) ) = a^{-n-1} \det \, ( B ) = \frac{ \kappa_K (x) }{ \Vert N_K ( x) \Vert^{n+1}_{K^\circ}} ,
\]
which is the result. 
\end{proof}

\begin{proof}[Proof of Theorem \ref{norm}]
We will use  formula (\ref{2homoasa}) for $\psi=\frac{\|\cdot\|_K^2}{2}$ and integrate in polar coordinates with respect to the normalized cone measure  $\sigma_K$  of  $K$. Thus, if we write $x=r\theta$, with $\theta\in\partial K$, $dx=n|K|r^{n-1}drd\sigma_K(\theta)$. We also use that  the map  $x\mapsto\det \, \Hess  \psi (x) $ is $0$-homogeneous. Therefore  we get with (\ref{2homoasa}), 
\begin{eqnarray*}
 as_\lam\left(\frac{\|\cdot\|_K^2}{2}\right)&=&n|K|\int_0^{+\infty}r^{n-1}e^\frac{-r^2}{2}dr \int_{(\partial K)_+}  \left(\det \,\Hess  \psi (\theta) \right)^\lam \ d\sigma_K(\theta)\\
 &=& (2\pi)^\frac{n}{2}\frac{|K|}{|B_2^n|}\int_{(\partial K)_+}   \left(\det \, \Hess \psi (\theta) \right)^\lam \ d\sigma_K(\theta).
\end{eqnarray*}
The relation between the normalized cone measure $\sigma_K$ and the Hausdorff measure $\mu_K$ on $\partial K$ is given by 
$$
d\sigma_K(x)=\frac{\langle x,N_K(x)\rangle d\mu_K(x)}{n|K|}.
$$
Observe  that for the function $G_K(x)$ introduced in Lemma  \ref{Hesse},  $\Vert G_K(x) \Vert_{K^\circ} =\langle x,N_K(x)\rangle$. Thus, with   $\lam=\frac{p}{n+p}$, 
\begin{eqnarray*}
as_\lam\left(\frac{\|\cdot\|_K^2}{2}\right)&=& \frac{(2\pi)^\frac{n}{2}}{n|B_2^n|}\int_{(\partial K)_+}  \left(\frac{ \kappa (x) }{ \langle x,N_K(x)\rangle^{n+1} }\right)^\lam \langle x,N_K(x)\rangle d\mu_K(x)\\
&=& \frac{(2\pi)^\frac{n}{2}}{n|B_2^n|}as_p(K), 
\end{eqnarray*}
when $\lam \in [0,1)$ or when $(\partial K)_+$ is of full Lebesgue measure in $\partial K$.
\end{proof}

Let us conclude this section with several observations.
First, observe that 
\[
\int e^{- \frac{\|x\|_K^2}{2}} dx = 2^\frac{n}{2} \Gamma\left(1 + \frac{n}{2}\right) |K|.
\]
Combining this with Theorem \ref{norm} and  Corollary \ref{funcsurfineq}, we recover  the known $L_p$-affine isoperimetric inequalities for convex bodies.  Namely, for a convex body $K$ with the origin in its interior, we get for $\lambda \in [0,1)$, which corresponds to $p \in [0, \infty)$ ($\lambda$ and $p$ are related via $\lambda=\frac{p}{n+p}$), 
\begin{eqnarray*}
\frac{as_p(K)}{as_p(B^n_2)}\leq \left(\frac{|K|}{|B^n_2
|}\right)^{\frac{n-p}{n+p}},
\end{eqnarray*}
with equality if and only if $K$ is an ellipsoid. 
For $\lambda \in (- \infty, 0]$, which corresponds to  $p \in  (-n, 0]$, we use Corollary \ref{cor:classic} and get that for any $C_2^+$ convex body $K$, 
\begin{eqnarray*}
\frac{as_p(K)}{as_p(B^n_2)}\geq
\left(\frac{|K|}{|B^n_2|}\right)^{\frac{n-p}{n+p}},
\end{eqnarray*} with
equality if and only if $K$ is an ellipsoid
and if $ \lambda \geq 1$, which corresponds to  $ p \in [- \infty , -n)$, then
\begin{equation*}
c^{\frac{np}{n+p}}\left(\frac{|K|}{|B^n_2
|}\right)^{\frac{n-p}{n+p}} \leq \frac{as_p(K)}{as_p(B^n_2 )}, 
\end{equation*}
where $c$ is a universal constant.
For $p\geq1$ these inequalities were proved by Lutwak  \cite{Lutwak1996} and for all other $p$ 
by Werner and Ye \cite{WernerYe2008}.
\par
Second,
the functional definition $as_\lam \left(\frac{\|\cdot\|_K^2}{2}\right)$ and $as_p(K)$ may not coincide for $p < 0$. Indeed, if $\partial K \setminus (\partial K)_+$ has non zero Lebesgue measure then $as_p(K) = + \infty$ while it can happen that the corresponding functional definition is finite. The simplest example is the convex hull of  the point $(-e_1)$ with the half unit sphere $\{\sum x_i^2 = 1, x_1 \ge 0\}$. 
\par
Note  that   $\left(\frac{\|\cdot\|_K^2}{2}\right)^* = \frac{\|\cdot\|_{K^\circ}^2}{2}$, where $K^\circ=\{y \in \mathbb{R}^n: \langle x, y \rangle \leq 1 \  \forall x  \in K\}$ is the polar body of $K$. Thus the functional duality relation (\ref{asa-dual}) implies the identity
\[
\forall \lam \in \R, \
as_\lam \left(\frac{\|\cdot\|_K^2}{2}\right) = as_{1-\lam} \left(\frac{\|\cdot\|_{K^\circ}^2}{2} \right).
\]
Together with Theorem \ref{norm} and taken $\lam = p/(n+p)$, we get the classical  duality relation 
\[
as_p(K) = as_{\frac{n^2}{p}}(K^\circ)
\]
for any  $p>0$. Moreover, this  is also valid for any $p \ne -n$ when $(\partial K)_+$ has full measure in $\partial K$. This duality relation was proved in \cite{Hug1996} for $p >0$  and for all $p\neq -n$ in \cite{WernerYe2008}, with some more regularity assumption when $p <0$.

\section{The $L_p$-affine surface area for  $s$-concave functions.} 
\label{asa-sconcave}
The purpose of this section is to generalize Definition \ref{def:log}, the functional version of $L_p$-affine surface area, to the context of $s$-concave functions for $s >0$.  
We could have defined $F_1 ( t ) = F_2(t) = F^{(s)}(t)=( 1 - s t )_+^{1/s}, $
where $a_+ = \max\{a,0\}$.
Since $F^{(s)}$ is log-concave and non-increasing, one has according to \eqref{eq:defofF}, $F = F^{(s)}$ and when $s\to 0$, it recovers the previous case of $F(t)=e^{-t}$.  However, when $\psi$ is convex, $F \circ \psi$ and $F \circ \psi^*$ are not satisfying a good duality relation. Instead of the Legendre duality, we follow in this section another point of view, coming from the duality introduced in \cite{ArtKlarMil} for $s$-concave functions.  

\subsection{The $s$-concave duality.}
We need few notations to explain the definition. 
Let $s \in (0,+\infty)$ and $f: \mathbb{R}^n \rightarrow \mathbb{R}_+$.  Following Borell \cite{Borell1975}, 
we say that $f$ is $s$-concave if for every $\lambda \in [0,1]$ and all  $x$ and $y$ such that $f(x) >0$ and $f(y) > 0$,
\[
f( (1-\lam)x + \lam y) \ge \left( (1-\lam) f(x)^s + \lam f(y)^s \right)^{1/s}.
\]
Since $s>0$, it is equivalent to assuming that $f^s$ is concave on its support. For the construction, we assume that $f$ is upper semi-continuous.
Let $S_f$ be the  convex set $\{x: f(x) >0\}$ and assume that $0$ belongs to the interior of $S_f$. This can be done by choosing correctly the origin of the space $\R^n$ and by assuming that $f$ is not a trivial function. This will not affect the construction. We   define the $(s)$-Legendre dual of $f$ as 
\[
f_{(s)}^\circ(y)  = \inf_{x \in S_f}  \frac{(1 - s \langle x, y \rangle)_+^{1/s}}{f(x)}.
\]
It coincides with the definition introduced in \cite{ArtKlarMil, ArtMil}.
Another point of view is to define a  function $\psi$ on $S_f$  by
\begin{equation}\label{def:psi}
\psi(x)= \frac{1-f^s(x)}{s},   \   \ x\in S_f.
 \end{equation}
and to associate a new dual function $\ps$ defined by 
\begin{equation}
\label{def: psi-stern}
\ps (y) =  \sup_{x \in S_f} \frac{\langle x, y \rangle - \psi(x)}{1 - s \psi(x)}   
\end{equation}
As $f > 0$ on $S_f$, $\psi$ is well defined and since $f$ is $s$-concave, $\psi$ is convex on $S_f$. Observe that 
$ \psi < \frac{1}{s}$, which means that  $1 - s \psi > 0$ on $S_f$. 
We can now define the $(s)$-Legendre dual of $f$ as 
\[
f_{(s)}^\circ(y) = \left(1- s \ps(y)\right)^{1/s}, \quad \forall y \in S_{f_{(s)}^\circ}
\]
where $S_{f_{(s)}^\circ} = \{ y, 1 - s \ps(y) > 0\}$.
By definition, $f_{(s)}^\circ$ is $s$-concave and upper semi-continuous. It is not difficult to see that as for the Legendre transform, 
 $(f_{(s)}^\circ)_{(s)}^\circ = f$ or equivalently that $(\ps)_{(s)}^{\star} = \psi$. 
Moreover, it can be seen that for $s > 0$,   $S_{f_{(s)}^\circ} = \frac{1}{s} S_f^\circ = \{ z, \forall x \in S_f, \langle x,z \rangle < 1\}$.
 
There is an implicit relation between the classical Legendre function $\psi^*$ and the $(s)$-Legendre function $\ps$ given by the formula
\begin{equation}
\label{implicitrelation} 
\forall y \in  S_{f_{(s)}^\circ}, \
\left( 1 - s \ps (y) \right) 
\left( 1 + s \psi^* \left( \frac{y}{1 - s \ps (y)} \right) \right)
= 1.
\end{equation}

Our definition in the $s$-concave case is the following.
\begin{defn}
\label{def:s}
For any $s > 0$,  let $f$ be an $s$-concave function and $\psi$ be the convex function associated above. For any $\lam \in \R$, let 
\[
as_{\lam}^{(s)}(\psi) = \frac{1}{1+ns} \  \int_{X_\psi} \frac{\left(1-s \psi(x) \right)^{\left(\frac{1}{s}-1\right)(1-\lam)}
 \left(\det   \Hess  \psi (x) \right)^\lam}
 {\left(1+s(\langle x , \nabla \psi(x) \rangle -\psi(x))\right)^{\lam\left(n+\frac{1}{s}+1\right) - 1 }} \  dx. 
\]
\end{defn}
It does not correspond to Definition \ref{def:log} with particulars function $F_1$ and $F_2$. As in the log-concave case, we call it the $L_\lam$-affine surface area  of  an $s$-concave function $f$. This is motivated by two main reasons. Like in Theorem \ref{generalduality} , we prove in Theorem \ref{s-generalduality} a satisfactory duality relation, from which we deduce a reverse log-Sobolev inequality for $s$-concave measures. Moreover, in the case $s= 1/k > 0$ where $k$ is an integer, this functional affine surface area corresponds to an $L_p$-affine surface area of a convex body build apart from $f$ in dimension $n+k$. Indeed, as in \cite{ArtKlarMil},  
we associate the convex body $K_{s}(f)$  in 
$\mathbb{R}^{n+\frac{1}{s}}$,
$$
K_{s}(f) = \left\{(x,y) \in \R^n \times \R^\frac{1}{s}: \frac{x}{\sqrt {s} }    \in
S_f,  \  |y| \leq f^s\left(\frac{x}{\sqrt {s} }  \right)\right\}.
$$
Then the $L_\lam$-affine surface area of $f$ is the  $L_p$-affine surface area of $K_{s}(f)$ with $p=\left(n+\frac{1}{s}\right)\ \frac{\lam}{1-\lam}$, 
\begin{equation}\label{motiv1}
(1+ns) \  as_{\lambda}^{(s)} (\psi)  = \frac{as_{p} \left((K_s(f) \right)}{s^\frac{n}{2} \mbox{vol}_{\frac{1}{s}-1} \left(S^{\frac{1}{s}-1}\right) }.
\end{equation}
Identity (\ref{motiv1})  follows from Proposition 5 in  \cite{CaglarWerner}.
\\
Finally, we note that, as it is the case  for log-concave functions, the $L_\lam$-affine surface area for $s$-concave functions  is  also  affine invariant under the action of $SL_n$ and has a degree of homogeneity.

\begin{theo} 
\label{s-generalduality} 
Let  $f$ be a an upper semi-continuous $s$-concave function with its corresponding convex function $\psi$. Assume that $0 \in S_f$. Let $\lam \in \R$ then
\[
as_{1-\lam}^{(s)}(\ps) = as_\lam^{(s)}(\psi).
\]
\end{theo}
\begin{proof}
Let us start with the case when $f$ is sufficiently smooth, say $f$ is twice continuously differentiable on $S_f$ and its Hessian is non zero. Then $\psi$ is
$\mathcal C^2_+$ on $\Omega_\psi$ and
\begin{equation}
\label{eq:1}
as_{\lam}^{(s)}(\psi) = \frac{1}{1+ns} \  \int_{\Omega_\psi} \frac{\left(1-s \psi(x) \right)^{\left(\frac{1}{s}-1\right)(1-\lam)}
 \left(\det   \Hess  \psi (x) \right)^\lam}
 {\left(1+s(\langle x , \nabla \psi(x) \rangle -\psi(x))\right)^{\lam\left(n+\frac{1}{s}+1\right) - 1 }} \  dx. 
\end{equation}
A simple computation tells that the supremum  in (\ref{def: psi-stern}) is attained at the point $x \in S_f$ such that
\begin{equation*}
\label{eq:attained}
y = \frac{1 - s \langle x, y \rangle}{1 - s \psi(x)} \, \nabla \psi(x)
\hbox{ which means } y = (1 - s \ps (y)) \nabla \psi (x). 
\end{equation*}
Therefore $\langle x,y \rangle = \frac{1 - s \langle x, y \rangle}{1 - s \psi(x)} \, \langle x, \nabla \psi(x)\rangle$ and we get that 
\begin{equation}
\label{eq:gradient}
\frac{1}{1 - s \ps(y)} = \frac{1- s \psi(x)}{1 - s \langle x,y \rangle} = 1 + s ( \langle \nabla \psi(x), x \rangle - \psi(x)).
\end{equation}
Finally, we get that 
\[
\ps (y) = \frac{\langle x, y \rangle - \psi(x)}{1 - s \psi(x)}  
\]
 if and only if
\[
y = \frac{\nabla \psi(x)}{1 + s ( \langle \nabla \psi(x), x \rangle - \psi(x))} = \frac{\nabla \psi(x)}{1 + s \psi^*( \nabla \psi(x))}.
\]
We define the change of variable 
\begin{equation}
\label{eq:changevariable}
\frac{\nabla \psi(x)}{1 + s ( \langle \nabla \psi(x), x \rangle - \psi(x))}  = T_\psi (x).\
\end{equation}
A straightforward computation shows that 
\[
d_x T_\psi = \frac{1}{1 + s \psi^*(\nabla \psi(x))}  \left( {\mathrm Id}  - \frac{s}{1 + s \psi^*(\nabla \psi(x))} x \otimes \nabla \psi(x) \right) \Hess \psi(x).
\]
Since 
\[
\det \left( {\mathrm Id}  - \frac{s}{1 + s \psi^*(\nabla \psi(x))} x \otimes \nabla \psi(x) \right)
=
1 - \frac{s}{1 + s \psi^*(\nabla \psi(x))} \langle x, \nabla \psi(x) \rangle
\]
we get that the the Jacobian of $T_\psi$ at $x$ is given by 
\begin{equation}
\label{eq:Jacobian2}
dy = \left| \det \, d_x T_\psi \right| dx =  \frac{1 - s \psi(x)}{\left(1 + s (\langle \nabla \psi(x), x \rangle - \psi(x))\right)^{n+1}} \ \det \Hess \psi(x) \ dx.
\end{equation}
As the  the duality $(\ps)_{(s)}^{\star} = \psi$ holds, we see that $T_{\psi}  \circ T_{\ps} =  {\mathrm Id}$ and $T_{\ps}  \circ T_{\psi} = {\mathrm Id}$ from which it is easy to deduce that for $y = T_\psi(x)$, 
\begin{equation}
\label{eq:JacobianDual}
\det \, \left(d_x T_{\psi}\right) \det \,  \left( d_{y} T_{\ps} \right)= 1.
\end{equation}
We make the change of variable $y = T_\psi(x)$ in formula \eqref{eq:1}. From \eqref{eq:gradient} and the fact that $(\ps)_{(s)}^{\star} = \psi$, we have 
\[
\frac{1}{1 - s \ps(y)} = 1 + s ( \langle \nabla \psi(x), x \rangle - \psi(x))
\hbox{ and }
\frac{1}{1 - s \psi(x)} = 1 + s ( \langle \nabla \ps(y), y \rangle - \ps(y)).
\] 
Combining with  \eqref{eq:Jacobian2} and \eqref{eq:JacobianDual} we get that 
\begin{equation}
\label{eq:dualityHessian}
\det \Hess \psi(x) \left( \frac{1 - s \ps(y)}{1 + s(\langle \nabla \ps(y), y \rangle - \ps(y)}  \right)^{n+2} \det \Hess \ps(y) = 1.
\end{equation}
Posing $y = T_\psi(x)$ we get 
\begin{align*}
(1+ ns) \ as_{\lam}^{(s)}(\psi) & =  \int_{\Omega_{\psi}} \frac{\left(1-s \psi(x) \right)^{\left(\frac{1}{s}-1\right)(1-\lam)-1}
 \left(\det   \Hess  \psi (x) \right)^{\lam-1}}
 {\left(1+s(\langle x , \nabla \psi(x) \rangle -\psi(x))\right)^{(\lam-1)(n+1) + \frac{\lam}{s} - 1 }} \  \left| \det \, d_x T_\psi \right| dx
 \\
 & = \int_{\Omega_{\ps}} \frac{\left(1-s \ps(y) \right)^{(n+2)(1-\lam) + (\lam - 1)(n+1) + \frac{\lam}{s} - 1}
 \left(\det   \Hess  \ps(y) \right)^{1-\lam}}
 {\left(1+s(\langle y , \nabla \ps(y) \rangle -\ps(y))\right)^{(n+2)(1-\lam) + \left( \frac{1}{s} - 1\right) (1 - \lam) - 1 }}  dy
 \\
 & = \int_{\Omega_{\ps}} \frac{\left(1-s \ps(y) \right)^{\lam \left(\frac{\lam}{s} - 1\right)}
 \left(\det   \Hess  \ps(y) \right)^{1-\lam}}
 {\left(1+s(\langle y , \nabla \ps(y) \rangle -\ps(y))\right)^{ \left( n + 1 + \frac{1}{s} \right) (1 - \lam) - 1 }}  dy
 \\
& = (1+ns) \ as_{1-\lam}^{(s)}(\ps).
\end{align*}
This concludes the proof in the smooth case. In the full generality, we need several observations. By \eqref{legendreequality}, we have a.e. in $\Omega_\psi$, 
\[
\left(1+s(\langle x , \nabla \psi(x) \rangle -\psi(x))\right) = 1 + s \psi^*(\nabla \psi(x))
\]
Therefore, we can use a result of Mc Cann \cite{mccann},  see \eqref{mccannbis}, to get 
\begin{align}
\nonumber
(1+ ns) \ as_{\lam}^{(s)}(\psi) & =  
\int_{X_\psi} \frac{\left(1-s \psi(x) \right)^{\left(\frac{1}{s}-1\right)(1-\lam)}
 \left(\det   \Hess  \psi (x) \right)^\lam}
 {\left(1+s \psi^*(\nabla \psi(x)) \right)^{\lam\left(n+\frac{1}{s}+1\right) - 1 }} \  dx
 \\
 \label{after1}
 & = \int_{X_{\psi^*}} \frac{\left(1-s \psi(\nabla \psi^*(z)) \right)^{\left(\frac{1}{s}-1\right)(1-\lam)}
 \left(\det   \Hess  \psi^* (z) \right)^{1-\lam}}
 {\left(1+s \psi^*(z) \right)^{\lam\left(n+\frac{1}{s}+1\right) - 1 }} \  dz.
 \end{align}
We make the change of variable  $z = T(y) = \frac{y}{1 - s \ps(y)}$.  Since $1 - s \ps$ is convex, it is not difficult to see that $T$ is an injective map. 
From \eqref{implicitrelation}, our change of variable is equivalent to 
$y = \frac{z}{1 + s \psi^*(z)}$. Therefore,
a.e. in $\Omega_{\ps}$, a similar computation to \eqref{eq:Jacobian2} gives 
\begin{equation}
\label{eq:Jacobian3}
|\det \ d_y T| = \frac{1 + s (\ps)^*(\nabla \ps(y))}{(1 - s \ps(y))^{n+1}}. 
\end{equation}
It can also be proved that it maps $X_{\psi}$ to $X_{\ps}$ and that the Alexandrov derivatives satisfy (this is similar to proposition A.1 in \cite{mccann}) 
\begin{equation}
\label{psips}
 \left( \frac{1 - s \ps(y)}{1 + s (\ps)^* (\nabla \ps(y))}  \right)^{n+2} \det \Hess \ps(y) = \det \Hess \psi^*(z).
\end{equation}
Since $(\ps)_{(s)}^{\star} = \psi$, we deduce from \eqref{implicitrelation} that 
\[
\forall x \in S_f, \
(1 - s \psi(x)) 
\left( 1 + s (\ps)^* \left( \frac{x}{1 - s \psi(x)} \right) \right) = 1
\]
Using \eqref{implicitrelation}  and the definition of $T$, it is not difficult to prove that a.e. in $\Omega_{\psi^*}$, 
\[
\frac{\nabla \psi^*(z)}{1 - s \psi(\nabla \psi^*(z))} = \nabla \ps (y), \ \hbox{ for } \ z = T y
\]
which shows that for $z = Ty$, 
\begin{equation}
\label{eq:dual-ps-psi}
1 - s \psi (\nabla \psi^*(z)) = \frac{1}{1 + s (\ps)^*(\nabla \ps (y))}.
\end{equation}
We have all the tools in hand to make the change of variable $z = T(y)$ in \eqref{after1} and to deduce from \eqref{eq:Jacobian3}, \eqref{psips}, \eqref{eq:dual-ps-psi} that 
\[
(1+ ns) \ as_{\lam}^{(s)}(\psi) 
=
\int_{X_{\ps}}  \frac{\left(1-s \ps(y) \right)^{\lam \left(\frac{\lam}{s} - 1\right)}
 \left(\det   \Hess  \ps(y) \right)^{1-\lam}}
 {\left(1+s(\langle y , \nabla \ps(y) \rangle -\ps(y))\right)^{ \left( n + 1 + \frac{1}{s} \right) (1 - \lam) - 1 }}  dy.
\]
This finishes the proof of the duality relation in the general case.
\end{proof}
\subsection{Consequences of the duality relation}
\label{sec:ConsDuality}
In this section, we suppose that $f$ satisfies more regularity assumptions: it is twice continuously differentiable on $S_f$,  its Hessian is non zero on $S_f$,  $\lim_{x \rightarrow \partial S_f} f^s(x)=0$ and recall that the origin belongs to the interior of $S_f$. With such assumptions, $X_\psi = S_f$ and $X_{\ps} = S_{f_{(s)}^{\circ}}$ and we remark that 
the definition of $as_{\lambda}^{(s)}(\psi)$ is made in such a way that 
\begin{equation}
\label{def:asa-0}
as_0^{(s)} (\psi) = \int_{S_f} f(x) dx 
\ \hbox{ and } \
as_1^{(s)} (\psi) = \int_{S_{f_{(s)}^{\circ}}} f_{(s)}^{\circ} (y) dy.
\end{equation}
Indeed, 
\begin{eqnarray*} 
as_{0}^{(s)}(\psi) &=& \frac{1}{1+ns} \  \int \left(1-s \psi(x)  \right)^{\frac{1}{s}-1}\left(1+s(\langle\grad \psi(x), x \rangle -\psi(x))\right)\  dx \nonumber \\
&= & \frac{1}{1+ns} \  \int  f(x) \left(1- s
\frac{\langle\grad f(x), x \rangle}{f(x)}\right)\  dx = \int f(x) dx, 
\end{eqnarray*}
where the last equality follows from Stokes formula and the fact that $\lim_{x \rightarrow \partial S_f} f^s(x)=0$. The second relation follows from the duality relation proved in Theorem \ref{s-generalduality}.  

In a way similar to the proof of Corollary  \ref{generalholder} and Theorem \ref{rev-log-sob}, it is possible to deduce from Theorem \ref{s-generalduality} some isoperimetric inequalities and a general reverse log-Sobolev inequality in the $s$-concave setting. 
 
\begin{prop} \label{Ungleichung1} Let  $f$ be a an $s$-concave function  that satisfies the regularity assumption defined at the beginning of Section 
\ref{sec:ConsDuality} and $\psi$ be  its associated convex function. Then
\[
\begin{split}
\forall \lambda \in [0,1], 
&\quad as_{\lambda}^{(s)} ( \psi )  \leq  \left(\int_{\R^n}f \  dx \right)^{1-\lambda}
 \left( \int_{\R^n}  f^\circ_{(s)}  \  dx \right)^\lambda  ; \\
\forall \lambda \notin [0,1], 
& \quad as_{\lambda}^{(s)} ( \psi )  \geq  \left(\int_{\R^n}f \  dx \right)^{1-\lambda}
 \left( \int_{\R^n}  f^\circ_{(s)}  \  dx \right) ^\lambda.
\end{split}
\]
\end{prop}
\begin{proof}
 We use H\"older inequality, (\ref{def:asa-0}) to prove the first inequality.   
\begin{eqnarray*}
as_{\lambda}^{(s)} ( \psi ) & \leq&  \frac{1}{1 + ns} \bigg[ \left( 
\int_{\R^n} \big( 1 - s \psi(x) \big)^{ \frac{1}{s} -1}  \  \big( 1 -  s \psi(x) + s \langle x, \nabla \psi(x)\rangle   \big) \  dx\right)^{1-\lambda}  \\ 
&& \hskip 29mm  \left( \int_{\R^n}  \frac{   \det\nabla^2 \psi(x)  }{\big(1 -  s \psi(x)  + s \langle x, \nabla \psi(x)\rangle   \big)^{n+ \frac{1}{s}} } \ dx  \right)^{\lambda} \bigg] \\
 &=& \left(\int_{\R^n}f \  dx \right)^{1-\lambda}
 \left( \int_{\R^n}  f^\circ_{(s)}  \  dx \right)^\lambda .
\end{eqnarray*}
Similarly, we use reverse H\"older inequality to prove the second inequality.
\end{proof}

The next theorem gives the log-Sobolev inequality for $s$-concave functions. There,  we put 
$$
d \mu =\left(1-s \psi\right)^{\left(\frac{1}{s}-1\right)}\left(1+s(\langle\grad \psi,x \rangle -\psi)\right) \  \frac{dx}{1+ns}.
$$
By (\ref{def:asa-0}),  $\mu$ is a probability measure on $\mathbb{R}^n$. 
We let $
S(\mu) = \int  -\log \left(\frac{d \mu}{dx }\right) d\mu$ be the Shannon entropy of $\mu$.
\begin{theo} 
\label{th:slogSob}  Let  $f$ be a an $s$-concave function  that satisfies the regularity assumption defined at the beginning of Section 
\ref{sec:ConsDuality} and $\psi$ be  its associated convex function. Assume moreover that   $f$ is even and  that $\int  f(x) dx =1$.  Then
\begin{eqnarray}\label{Ungleichung1000}
\int \log \left(\det \left( \Hess \psi(x) \right)  \right) d\mu  \leq \int \log \left( \left(1+s(\langle x, \nabla \psi(x) \rangle -\psi(x))\right)^{\frac{1}{s } +n } \right) d\mu  -  S(\mu)\nonumber 
 \\
+ \log \left(  \left(\frac{\pi}{s}\right)^n  \frac{(1+ns) \left(\Gamma(1+\frac{1}{2s})\right)^2}{ \left(\Gamma(1+\frac{n}{2}+\frac{1}{2s})\right)^2} \right).
\end{eqnarray} 
There is equality if and only if  there is a positive definite matrix $A$ such that $f(x) = c_0  \left(1-  s \left| Ax\right|^2\right)^\frac{1}{2s}$, where $
 c_0  =  \left(\frac{\pi}{s}\right)^{-\frac{n}{2}} \left(  \frac{\Gamma(1+\frac{1}{2s})}{\Gamma(1+\frac{n}{2}+\frac{1}{2s})}\right)^{-1} $.
\end{theo}
\noindent
{\bf Remark.}   $S(\gamma_n)= \log\left(2 \pi e\right)^\frac{n}{2}$.  Therefore,   the right hand side the inequality (\ref{Ungleichung1000}) tends to $2 \left[ S(\gamma_n)- S(\mu)\right]$
for $s\rightarrow 0$ 
and we recover the inequality of Theorem \ref{rev-log-sob}.

\begin{proof} 
The proof follows the line of the proof of Theorem \ref{rev-log-sob} presented in Section \ref{sec:shortproof}. 
By the definition (\ref{def: psi-stern}) of $\ps$, we have for all $x \in S_f$ and  for all $y \in \frac{S_f^\circ}{s}$ that
\[
f(x) f_{(s)}^{\circ}(y) = (1- s\psi(x)) ^\frac{1}{s} \  (1-s \ps(y)) ^\frac{1}{s} \leq \left( 1 - s \langle x, y \rangle \right)^\frac{1}{s}.
\]
We let $\rho(t) = (1-st)_+^\frac{1}{2s}$. As $f \equiv 0$ outside $S_f$ and $f_{(s)}^{\circ} \equiv 0$ outside $\frac{S_f^\circ}{s} $,   the functions $f$ and $f_{(s)}^{\circ}$ satisfy the assumption \eqref{hypsantalo} with $z_0 = 0$ because $f$ is even. It follows from \eqref{concsantalo}
that 
\begin{equation}
\label{eq:Santalo2}
\left(\int f dx \right) \left( \int f_{(s)}^{\circ} dx \right) \leq \left( \int (1-s |x|^2)_+^\frac{1}{2s}dx \right) ^2= \left(\frac{\pi}{s}\right)^n  \frac{ \left(\Gamma(1+\frac{1}{2s})\right)^2}{ \left(\Gamma(1+\frac{n}{2}+\frac{1}{2s})\right)^2}.
\end{equation}
By Theorem \ref{s-generalduality}, we have $\int f_{(s)}^{\circ} = as_0^{(s)} (\psi_{(s)}^\star) = as_1^{(s)} (\psi)$ which means that
\begin{align*}
\int f_{(s)}^{\circ}  =  &
\frac{1}{1+ns}   \int_{X_\psi} \frac{ \det   \Hess  \psi (x) }
{\left(1+s(\langle x , \nabla \psi(x) \rangle -\psi(x))\right)^{\left(n+\frac{1}{s}\right) }} \  dx
\\ = & 
\frac{1}{1+ns}   \int_{X_\psi} \frac{ \det   \Hess  \psi (x) }
{\left(1+s(\langle x , \nabla \psi(x) \rangle -\psi(x))\right)^{\left(n+\frac{1}{s}\right) }} \  \frac{dx}{d\mu(x)} \ d\mu(x)
\end{align*}
Since $\int f = 1$, $\mu$ is a probability measure and we get from Jensen inequality
\begin{eqnarray*} 
\log\left( \int f^\circ_{(s)}\right)  &\geq& S(\mu)  - \log(1+ns) + \int \log \left(\det \Hess \psi\right)  d\mu  \nonumber \\
&-& \int \log \left( \left(1+s(\langle x, \nabla \psi(x) \rangle -\psi)\right)^{\frac{1}{s } +n } \right) d\mu.
\end{eqnarray*}
Therefore, with \eqref{eq:Santalo2}  and as $\int f dx  =1$, 
\begin{eqnarray*}
\int \log \left(\det \left(\Hess \psi\right)  \right) d\mu  \leq \int \log \left( \left(1+s(\langle x, \nabla \psi(x) \rangle -\psi)\right)^{\frac{1}{s } +n } \right) d\mu  -  S(\mu) \\
 \\+  \log(1+ns) 
+ \log \left( \left(\frac{\pi}{s}\right)^n \frac{\left(\Gamma(1+\frac{1}{2s})\right)^2}{\left(\Gamma(1+\frac{n}{2}+\frac{1}{2s})\right)^2} \right).
\end{eqnarray*}
When equality holds in (\ref{Ungleichung1000}), then in particular equality holds in the Blaschke Santal\'o inequality \eqref{eq:Santalo2}.
It was proved in \cite{FradeliziMeyer2007} that this
happens if and only if, in our situation,  $f(x) = c_0 \left(1-  s \left|Ax\right|^2\right)^\frac{1}{2s}$, for a positive definite matrix $A$ and where $c_0$  as above is  chosen such that $\int f dx =1$. 
On the other hand,  
it is easy to see that equality holds   in  (\ref{Ungleichung1000}), when  $f(x) = c \left(1-  s \left|Ax\right|^2\right)^\frac{1}{2s}$, for a positive definite matrix $A$ and a positive constant $c$.
\end{proof}

\noindent 
Umut Caglar
\\
{\small Department of Mathematics } \\
{\small Case Western Reserve University } \\
{\small Cleveland, Ohio 44106, U. S. A. } \\
{\small \tt umut.caglar@case.edu}\\ \\
\and
Matthieu Fradelizi and Olivier Gu\'edon
\\  
{\small Universit\'e Paris Est } \\
{\small Laboratoire d'Analyse et de Math\'ematiques Appliqu\'ees (UMR 8050)} \\
{\small UPEM, F-77454, Marne-la-Vall\'ee, France} \\
{\small \tt matthieu.fradelizi@u-pem.fr, olivier.guedon@u-pem.fr}\\ \\
\and
Joseph Lehec
\\
{\small Universit\'e Paris-Dauphine } \\
{\small CEREMADE } \\
{\small Place du Mar\'echal de Lattre de Tassigny,
75016 Paris, France } \\
{\small \tt  lehec@ceremade.dauphine.fr}\\ \\
\and
Carsten Sch\"utt\\
{\small Mathematisches Institut}\\
{\small Universit\"at Kiel}\\
{\small 24105 Kiel, Germany}\\
{\small \tt schuett@math.uni-kiel.de }\\ \\
\noindent
\and 
Elisabeth Werner\\
{\small Department of Mathematics \ \ \ \ \ \ \ \ \ \ \ \ \ \ \ \ \ \ \ Universit\'{e} de Lille 1}\\
{\small Case Western Reserve University \ \ \ \ \ \ \ \ \ \ \ \ \ UFR de Math\'{e}matique }\\
{\small Cleveland, Ohio 44106, U. S. A. \ \ \ \ \ \ \ \ \ \ \ \ \ \ \ 59655 Villeneuve d'Ascq, France}\\
{\small \tt elisabeth.werner@case.edu}\\ \\

\end{document}